\documentclass{article}
\usepackage{fullpage}
\usepackage{amsmath}
\usepackage{amsthm}
\usepackage{amssymb}
\usepackage{url}
\usepackage{graphicx}
\usepackage{verbatim}
\usepackage{listings}
\usepackage{url}
\usepackage{graphicx}
\usepackage{tikz}
\usetikzlibrary{calc}
\usetikzlibrary{patterns}
\tikzstyle{mybox} = [draw=black, fill=white,  thick,
    rectangle, inner sep=10pt, inner ysep=20pt]
\tikzstyle{mybox} = [draw=black, fill=white,  thick,
    rectangle, inner sep=2pt, inner ysep=2pt]

\newtheorem{thm}{Theorem}

\theoremstyle{definition}
\newtheorem{definition}{Definition}
\newtheorem{remark}{Remark}

\def\Cal{\mathcal}

\def\br {{\mathbb R}} 
\def\bh {{\mathbb H}} 
\def\bc {{\mathbb C}} 
\def\al {\alpha}

\begin{document}


\title{Algorithms and Polynomiography for Solving Quaternion Quadratic Equations}
\author{Fedor Andreev\thanks{Department of Mathematics,
        Western Illinois University, {\tt F-Andreev@WIU.EDU}}
        ~and Bahman Kalantari\thanks{Department of Computer Science,
        Rutgers University, {\tt kalantari@cs.rutgers.edu}}}
\date{}
\maketitle

\date{}
\maketitle



\begin{abstract}
Solving a quadratic equation $P(x)=ax^2+bx+c=0$  with real coefficients is known to middle school students.  Solving the equation over the quaternions is not straightforward. Huang and So \cite{Huang} give a complete set of formulas, breaking it into several cases depending on the coefficients.  From a result of the second author in \cite{kalQ}, zeros of $P(x)$ can be expressed in terms of the zeros of a real quartic equation. This drastically simplifies solving a quadratic equation. Here we also consider solving $P(x)=0$ iteratively via Newton and Halley methods developed in \cite{kalQ}. We prove a property of the Jacobian of Newton and Halley methods and describe several 2D {\it polynomiography} based on these methods. The images not only encode the outcome of the iterative process, but by measuring the time taken to render them we find the relative speed of convergence for the methods.

\end{abstract}




\section{Introduction} In this article we are interested in solving a quaternion quadratic  equation
\begin{equation}
P(x)=ax^2+bx+ c=0, \quad a, b, c \in \mathbb{H}, \quad a \not =0
\end{equation}
where $\mathbb{H}$ is the set of quaternions, an extension of the
complex field $\mathbb{C}$, first described by Sir William Rowan
Hamilton in 1843:
\begin{equation}
\mathbb{H}=\{q=a_1+a_2{\bf i}+a_3{\bf j}+a_4{\bf k}~|~ a_i \in \mathbb{R} \},
\end{equation}
where $\mathbb{R}$ is the reals and
\begin{equation}
{\bf i}^2={\bf j}^2={\bf k}^2={\bf ijk}=-1.
\end{equation}
These imply
\begin{equation}
{\bf ij}=-{\bf ji}={\bf k},~~~{\bf jk}=-{\bf kj}={\bf i},~~~{\bf ki}=-{\bf ik}={\bf j}.
\end{equation}
The set of quaternions  is a non-commutative division ring. It can
be identified with $\mathbb{R}^4$
\begin{equation}
a_1+a_2{\bf i}+a_3{\bf j}+a_4{\bf k} \Longleftrightarrow (a_1,a_2,a_3,a_4).
\end{equation}

The variable $x$ is assumed to commute with the coefficients, however at a given quaternion $q$ the evaluation of the polynomial is defined as

\begin{equation}
P(q)=aq^2+bq+c.
\end{equation}
If $P(q)=0$ we say $q$ is a root. Since $a \not=0$ we can multiply by its inverse to get $q^2+a^{-1}bq+a^{-1}c=0$.
Thus  without loss of generality we assume $a=1$ so that the monic equation of interest is
\begin{equation} \label{quadeqn}
P(x)=x^2+bx+c=0.
\end{equation}

Unlike the complex case, over the quaternions solving a quadratic
equation is nontrivial. Huang and So \cite{Huang} give complete
formulas for a quadratic equation $x^2+bx+c=0$, breaking it into
several cases depending on the constants $b$ and $c$. Applications of
this special case of quaternion equation are described in Heidrich
\cite {Heidrich} and Huang and So \cite{Huang2}. Others have studied
solution of quaternion quadratic equations such as Niven
\cite{Niven}, Zhang and Mu \cite{Zhang2}. The quadratic equations
where $x$ and the coefficients may not commute  is harder to
solve. Porter \cite{Porter} considers such a case, e.g. $x^2
+axb+cxd+e=0$ where a solution is already known. A method for solving
more general non-unilateral quadratic equations is described in Jia
et al \cite{Jia}.

Our interest here lies in solving quaternion quadratic
equation (\ref{quadeqn}). Recently, in \cite{kalQ} several algorithms for solving a general quaternion polynomial
equation are described. In particular, for a quaternion quadratic it  gives new
algorithms for computing the solutions, one of which drastically
simplifies the Huang and So \cite{Huang} formulas. This is in the
sense that the solutions to $P(x)=0$ are directly related to a
quartic equation with real coefficients for which one may use existing formulas. In particular, we can
apply any complex polynomial root-finding algorithm. We will
give some computational results as well as their polynomiography.
Polynomiography refers to algorithmic visualization of a polynomial equation, see \cite{kalbook}-\cite{kal2005c}. In
other algorithms we develop Newton and Halley methods directly
working on $P(x)$ in the quaternion space. We will present computational
results and 2D polynomiography with these algorithms.

Next we review some basic properties of the quaternions. Some
relevant references are Lam \cite{Lam}, Niven \cite{Niven}, and
Zhang \cite{Zhang}. The {\it conjugate} of a quaternion $q=a_1+{\bf
i}a_2+{\bf j}a_3+{\bf k}a_4$ is defined as
\begin{equation}
\overline q =a_1-a_2{\bf i}-a_3{\bf j}-a_4{\bf k}.
\end{equation}
The number $a_1$ is the {\it real part} of $q$. The {\it trace} of $q$ is
\begin{equation}
t(q)=q+ \overline q.
\end{equation}
The {\it norm} of $q$ is
\begin{equation}
\nu(q)=|q | = \sqrt{q \overline q} = \sqrt{\overline q q} = \sqrt{a_1^2+a_2^2+a_3^2+a_4^2}.
\end{equation}
The {\it inverse} of a nonzero quaternion $q$  is the unique quaternion denoted by $q^{-1}$ such that $q q^{-1}=q^{-1}q=1$.
It follows that
\begin{equation}
q^{-1}=  \frac{\overline q }{| q |^2}.
\end{equation}
Given quaternions $q_1, q_2$ we have
\begin{equation}
\overline {q_1+q_2}=\overline {q_1} + \overline {q_2}, \quad \overline{q_1q_2}= \overline {q_2} \cdot  \overline{q_1}, \quad |q_1q_2|= |q_1||q_2|.
\end{equation}
The division of a quaternions $q_1$ by $q_2$ must be specified either as  $q_1  q_2^{-1}$ or as $q_2^{-1}  q_1$.
It can be shown that when  $q_1, q_2$ are nonzero then
\begin{equation}
(q_1q_2)^{-1}=q_2^{-1}q_{1}^{-1}.
\end{equation}

Two quaternions $q$ and $q'$ are said to be {\it congruent} or {\it
equivalent}, written $q \sim q'$, if for some  quaternion $w \not
=0$ we have $q'=wqw^{-1}.$ The {\it congruent class} of $q=a_1+a_2{\bf
i}+a_3{\bf j}+a_4{\bf k}$, denoted by $[q]$ is the set of all
quaternions congruent to $q$. It can be shown

\begin{equation} \label{charcter}
[q]=\{a_1+x_2{\bf i}+x_3{\bf j}+x_4{\bf k}~|~ x_2^2+x_3^2+x_4^2=a_2^2+a_3^2+a_4^2\}.
\end{equation}
From the above it follows that  $[q]$ is a singleton element if and
only if $q$ is a real number. If $q$ is not real its congruent class
is the three-dimensional sphere in $ \mathbb{R}^4$ centered at
$(a_1,0,0,0)$ having radius $\sqrt{a_2^2+a_3^2+a_4^2}$. It
follows that any quaternion is congruent to a complex number with
the same real part and norm, specifically $q \sim a_1+ {\bf i}
\sqrt{a_2^2+a_3^2+a_4^2}$.

The {\it characteristic polynomial} of a non-real quaternion $q=a_1+a_2{\bf i}+a_3{\bf j}+a_4{\bf k}$ is
\begin{equation}
P_q(x)=x^2 - t(q)x+ \nu^2(q).
\end{equation}

From (\ref{charcter}) it follows that for any $q' \in [q]$ we have
\begin{equation} \label{Char2}
P_{q'}(x)=x^2 - t(q')x+ \nu^2(q')=(x-q')(x-\overline{q'})=P_q(x).
\end{equation}

The {\it discriminant} of $P_q(x)$ is
\begin{equation}
\Delta=t^2(q)-4\nu^2(q)=4a_1^2-4(a_1^2+a_2^2+a_3^2+a_4^2)=-4(a_2^2+a_3^2+a_4^2) <0.\end{equation}
The set of zeros of $P_q(x)$ is $[q]$. In particular, the quadratic
polynomial $x^2+1$ is the characteristic polynomial of $\pm {\bf
i}$, $\pm {\bf j}$, $\pm {\bf k}$, its zeros constitute the unit
sphere centered at the origin $(0,0,0,0)$:
\begin{equation}
\{x_2{\bf i}+x_3{\bf j}+x_4{\bf k} \in \mathbb{H}~|~ x_2^2+x_3^2+x_4^2 =1\}.
\end{equation}

Conversely, given reals $t \geq 0$ and $\nu >0$ satisfying $t^2-4
\nu^2 <0$, the polynomial $G(x)=x^2-tx+\nu^2$ has roots
$\theta=\frac{1}{2}(t + i\sqrt{t^2-4 \nu^2})$ and its conjugate
$\overline \theta$. The set of roots of $G(x)$ is $[\theta]$.

The infinitude of zeroes is only one of the peculiarities with quaternion polynomials. Another one is a possibility of having less roots than the degree.
Consider $x^2-({\bf i}+{\bf j})x+{\bf k}$. It can be shown that
this equation has only one root, ${\bf j}$ with multiplicity  one,
see \cite{kalQ}.

In this article we consider quadratic quaternion equations and algorithms for solving them, as well as their polynomiography.  In Section 2, we review two theorem that give exact and approximate solution of $P(x)$ in terms of those of a real quartic polynomial. In Section 3, we describe a Taylor's theorem for a quadratic polynomial, giving rise to a Newton and Halley methods. In Section 4, we prove a property of the Jacobian of Newton and Halley methods when applied to a quadratic quaternion. This results in the definition of a {\it local invariant plane} for these iteration functions. In Section 5, we describe various ways for doing polynomiography with Newton and Halley methods applied to a quadratic quaternion. We describe various ways for selecting proper rectangular subsets in the quaternion space. Then the corresponding polynomiograph is rendered where each {\it initial iterate} on it goes through a sequence of {\it intermediate iterates} in the quaternion space, or projections of these iterates, before it generates a {\it terminal iterate}. In particular, we give algorithms and polynomiography based on four methods that trace: (I) quaternion iterates; (II) 2D projection of intermediate iterates, (III) 2D
projection of congruent intermediate iterates, and (IV)  iterates in locally invariant planes. All computations are carried out with Mathematica.

\section{Solutions of Quaternion Quadratic Via a Real Quartic} Given $P(x)=x^2+bx+c$, its {\it quaternion conjugate} polynomial, or just {\it conjugate} is

\begin{equation}
\overline P(x)= x^2+ \overline b x + \overline c.
\end{equation}
Let
\begin{equation}
F(x)=P(x)\overline P(x)= x^4+(b +\overline b)x^3+(c + \overline c + b \overline b)x^2 +(b \overline c + c \overline b)x+c \overline c.
\end{equation}

Clearly $F(x)$ is a quartic real polynomial. The
following relates its solutions to those of $P(x)$:

\begin{thm} \label{thmFx} {\rm (\cite{kalQ})} Let $\theta \in \mathbb{C}$ be a root of $F(x)$, then $P(x)$ has a root in $[\theta]$.
Furthermore, if $q \in \mathbb{H}$ is a root of $P(x)$, then there exist a complex number $\theta \in [q]$ which is a root of $F(x)$. More specifically,

(i) Suppose  $\theta \in \mathbb{C}$  is a root of
$F(x)$.  If $\theta$ is not a root of $\overline P (x)$, then
\begin{equation}
\overline P(\theta) \theta \overline P(\theta)^{-1}
\end{equation}
is a root of $P(x)$.

(ii) Suppose $\theta \in \mathbb{C}$ is a root of
$F(x)$. If $\theta$ is a root of $\overline P (x)$, and $\overline
\theta$ is not a root of $\overline P (x)$, then
\begin{equation}
\overline P(\overline \theta) \overline \theta ~\overline P(\overline \theta)^{-1}
\end{equation}
is a root of $P(x)$.

(iii) Suppose $\theta \in \mathbb{C}$ is a root of $F(x)$. If $\theta$  and $\overline \theta$ are both roots of $\overline P (x)$,  then both $\theta$ and $\overline \theta$ are also roots of $P(x)$. In particular, if $\theta \in \mathbb{R}$, then $P(\theta)=0$.\\

(iv) Suppose $q=a_1+a_2{\bf i}+a_3{\bf j}+a_4{\bf k}$
is a root of $P(x)$. Then $\theta=a_1+{\bf i}
\sqrt{a_2^2+a_3^2+a_4^2}$ is a root of $F(x)$. $\square$
\end{thm}

\subsection{Approximate Zeros of a Quaternion Polynomial}

\begin{thm} \label{Approx} {\rm (\cite{kalQ})}Let $\epsilon \in (0,1)$. Suppose $ \theta \in \mathbb{C}$ satisfies
\begin{equation}
|F(\theta)| \leq \epsilon^2.
\end{equation}
Let
\begin{equation}
q=\overline P(\theta) \theta \overline P(\theta)^{-1}, \quad q'=\overline P( \overline \theta)  \overline  \theta ~\overline P( \overline \theta)^{-1}.
\end{equation}

(i): If $|\overline P(\theta)| > \epsilon$, then $|P(q)| < \epsilon$.\\

(ii): If $|\overline P( \overline \theta)| > \epsilon$, then $|P(q')| < \epsilon$.\\

(iii): If $|\overline P(\theta)| < \epsilon$ and  $|\overline P(
\overline \theta)| < \epsilon$, then $|P(\theta)| < \sqrt{2} \epsilon$
and $|P(\overline \theta)| < \sqrt{2} \epsilon$. $\square$
\end{thm}

From the above we see that by solving the quartic equation
$F(x)=0$ exactly or approximately we can recover the solutions of
$P(x)=0$. This can be achieved via closed formulas or numerical
methods such as Newton and  Halley methods, or any member of
an infinite family of iteration functions, see \cite{kalbook}.

\section{Taylors's Theorem, Newton and Halley Methods for Quadratics} \label{sectionVIII}

In \cite{kalQ} a Taylor's Theorem for quaternion polynomials is derived. This theorem is then used to develop Newton method for general quaternion polynomials.  We state the theorem for
a quadratic  quaternion polynomial $P(x)=x^2+bx+c$. Define its first and second derivatives as:

\begin{equation}
P'(x)=2x+b,  \quad P''(x)=2.
\end{equation}

\begin{thm} \label{Expansion} Let $\xi$ be a root of $P(x)=x^2+bx+c$ and $q$ any quaternion. Then
\begin{equation} \label{T15eqn.p1}
E(\xi, q) \equiv \sum_{k=0}^2 {{P^{(k)}(q)} \over {k!}} (\xi - q)^k= (q \xi - \xi q).
\end{equation}
In particular, if $q$ commutes with $\xi$,  $E(\xi, q)=0$. $\Box$
\end{thm}

This gives rise to the definition of Newton method and a corresponding expansion.

\begin{thm} \label{Newt}  Suppose $\xi$ is a simple root of $P(x)$. Given $q$ with $P'(q) \not =0$,
\begin{equation} \label{T15eqn.p3}
N(q) \equiv q- P'(q)^{-1}P(q)=
 \xi + P'(q)^{-1} (\xi -q)^2 - P'(q)^{-1} E(\xi, q). \quad \square
\end{equation}
\end{thm}

\begin{definition} \label{defNewt}
Given a seed $q_0$, let the {\it Newton fixed point iteration} be
\begin{equation}
q_{k}=N(q_{k-1})=q_{k-1}- P'(q_{k-1})^{-1}P(q_{k-1}), \quad k \geq 1.
\end{equation}
\end{definition}

The iteration is well-defined everywhere except for the solution to $P'(x)=0$.
For a complex polynomial $N$ defines Newton method, the first member of the infinite Basic Family of iterations functions, see \cite{kalbook}. The behavior of Newton method for a real or quadratic complex polynomial is well-understood: if there are two distinct roots  the basin of attraction of each root is the Voronoi region of the root, i.e. the set of all points in the Euclidean plane that are closer to this root than the other. However, over the quaternions even for a quadratic polynomial Newton method could exhibit a strange behavior. We consider an example from \cite{kalQ}. Suppose that $P(x)=x^2-({\bf i}+{\bf j})x+{\bf k}$. It can be shown that ${\bf j}$ is the only solution of $P(x)=0$. If we set $q_0={\bf j}+ \epsilon$, then it can be seen that for the next iterate we have, $|q_1- {\bf j}| \approx \epsilon^2 = |q_0- {\bf j}|^2$. Now suppose that $q_0={\bf j}+  \epsilon_1 + \epsilon_2 {\bf i}+ \epsilon_3 {\bf k}$. Then $|q_0- {\bf j}|= \sqrt{\epsilon_1^2+\epsilon_2^2+\epsilon_3^2}$. It can be shown that
we can select $\epsilon_1,  \epsilon_2, \epsilon_3$ so that $|q_1-{\bf j}| < |q_0 - {\bf j}|$. However, we can also select these so that $|q_1-{\bf j}| > |q_0 - {\bf j}|$. Hence in every neighborhood of the root ${\bf j}$ there is a repulsive direction, i.e. a direction where the Newton iterate gets farther away from the root than the current iterate.  This is contrary to the case of a complex polynomial where each root is necessarily an attractive fixed point of Newton iteration function. This implies the behavior of Newton or other iterations could be chaotic even in a neighborhood of a root.

Janovska and Opfer \cite{{Janovska}} consider solving quaternionic roots by Newton method for the equation $x^n-a$ and consider the convergence case of a formal Newton method and local convergence.  The above example shows even for a quadratic the performance could be chaotic.

\subsection{Halley Method for  Quadratic Quaternion Polynomials}
We consider Halley method for quadratic quaternion polynomials. Using the results in \cite{kalQ} we have

\begin{thm} For a $q$ assume $P(q) \not=0$, $P'(q) \not =0$, $P'(q) \not =1$.  Then,
\begin{equation} \label{Halley}
H (q) \equiv q-P'^{-1}P-P'^{-1}\Delta^{-1} P P'^{-1}P =   - P'^{-1} \Delta^{-1} (\xi -q)^3 + E_3,
\end{equation}
where, $\Delta=  ( P'  -  P P'^{-1})$ and
\begin{equation}
E_3=  P'^{-1} \Delta^{-1} \bigg ( P P'^{-1}E(\xi,q)-E(\xi, q)(\xi-q) \bigg)-  P'^{-1}E(\xi,q). \quad \square
\end{equation}
\end{thm}

\begin{definition}
Given a seed $q_0$, let the {\it Halley fixed point iteration} be
\begin{equation}
q_{k}=H(q_{k-1}), \quad k \geq 1.
\end{equation}
\end{definition}

\section{Local Invariant Plane for Newton and Halley Methods}
\label{LIPlane}
Here we consider the quadratic quaternion $P(x)=x^2+bx+c$ with two roots $\alpha$ and $\beta$ from different conjugacy classes.  We consider three iterative methods for such a quadratic which act on a quaternion $q$ as follows:

\noindent (i) Left-Newton method
\begin{equation}
 N_l(q)=q- P'(q)^{-1} P(q),
\end{equation}
(ii) Right-Newton method
\begin{equation}
N_r(q)=q- P(q)P'(q)^{-1},
\end{equation}
(iii) Halley method
\begin{equation}
H(q)=q- (P'(q))^{-1} P(q) - P'(q)^{-1} \Delta^{-1}P(q) P'(q)^{-1}P(q),
\quad \Delta= P'(q)-P(q) P'(q)^{-1}.
\end{equation}
Each root of $P(x)$ is  a fixed point of all three iteration functions. Let $f$ denote any of the three iteration functions. Since
$\mathbb{H}$ can be identified with $\mathbb{R}^4$,
$f$ can be viewed as a mapping from $\mathbb{R}^4$ into $\mathbb{R}^4$. Let
\begin{equation}
D(q)= (d_{ij})= (\partial f_i(q)/ \partial q_j)
\end{equation}
denote the $4 \times 4$ Jacobian matrix for this mapping at $q$.

\begin{thm} \label{inv} Let $\alpha$ be a root of $P(x)=x^2+bx+c$, then the determinant of $D(\alpha)$ is zero. More specifically, the rank of $D(\alpha)$ is at most $2$.
\end{thm}

\begin{proof}   We show that there are two linearly independent directions in which the iteration function $f$ does not change in its linear approximation.  Let $u$ be a quaternion that commutes with $\alpha$.
We will show that the partial derivative $f$ in the direction of $u$ is zero.

Let $q= \alpha +r u$, where $r$ is a real number.  Consider $P(q)$ and $P'(q)$.  Let $v= P'(\alpha)=2 \alpha +b$. Using that $u$ and $\alpha$ commute we have:
\begin{equation}
P(q)=r^2u^2+2r\alpha u+ r bu,  \quad P'(q)=v+ 2r u=v (1+ 2rv^{-1}u).
\end{equation}
We may write
\begin{equation}
P(q)=r^2u^2+rvu,
\end{equation}
and
\begin{equation}
P'(q)^{-1}= (1+ 2rv^{-1}u)^{-1}v^{-1}.
\end{equation}
Let
$w=2v^{-1}u$, and consider $(1+r w)^{-1}$ when $r$ is sufficiently small.  It can be shown that in such case we may write
\begin{equation}
(1+ r w)^{-1}= 1 - rw +r^2w^2 - r^3w^3 + \cdots = 1-rw + O(r^2).
\end{equation}
Thus
\begin{equation}
P'(q)^{-1}= (1+ 2rv^{-1}u + O(r^2))v^{-1}.
\end{equation}
This gives
\begin{equation}
P'(q)^{-1}P(q)= (1+ 2rv^{-1}u + O(r^2))v^{-1} (r^2u^2+rvu)=
\end{equation}
\begin{equation}
(1+ 2rv^{-1}u + O(r^2))(r^2v^{-1}u^2+ru)= ru- r^2v^{-1}u^2 + O(r^3).
\end{equation}
This implies if $f(q)=N_l(q)$ (Left-Newton), then
\begin{equation}
f(q)=\alpha + ru - ru- r^2v^{-1}u^2 + O(r^3)= \alpha +  O(r^3).
\end{equation}
Thus
\begin{equation}
f(\alpha + r u)=f(\alpha) +O(r^2).
\end{equation}
This implies that the partial derivative $f$ in the direction of $u$ is zero.  Since we can find at least two such directions $u$ that commute with $\alpha$, this proves the theorem for Left-Newton. For instance, if $\alpha$ is not real we can use the direction of $\mathbb{R}$ and $\alpha \mathbb{R}$. If $\alpha$ is real all four directions commute with it.  Thus in at least two directions the derivative is zero.

For the Right-Newton we write
\begin{equation}
P'(q)=v+ 2r u= (1+ 2ruv^{-1})v.
\end{equation}
We assume now $u$ is such that it commutes with $\alpha +b$. This allows to write
\begin{equation}
P(q)=ru(2\alpha +b + r u).
\end{equation}
Then
\begin{equation}
P(q) P'(q)^{-1}= ru(v+ru)v^{-1}(1+2ruv^{-1})^{-1}= ru-r^2 u^2 v^{-1}+
O(r^3).
\end{equation}
Then if $f(q)=N_r(q)$, we have
\begin{equation}
f(q)= \alpha + r u-ru+r^2u^2v^{-1}+O(r^3)= \alpha + O(r^2).
\end{equation}
Again the commutivity  condition $u(\alpha+b)=(\alpha+ b)u$ allows selecting two possible independent directions,  $\mathbb{R}$ and $(\alpha+b) \mathbb{R}$, along which the derivative of Right-Newton is zero.

For Halley method, analogous to Left-Newton, it suffices to choose $u$ a direction that commutes with $\alpha$.  The Halley method coincides with Left-Newton in the first-order approximation.
\end{proof}

\begin{remark}
It follows that the dimension of the eigenspace corresponding to $\lambda=0$ eigenvalue is 2. (Strictly speaking: at least two, but for simplicity we will be omitting this caveat in the following discussion.) There are two vectors that span this eigenspace. The first one is a pure-real vector $(1,0,0,0)$. The second zero mode at the root $\alpha$ corresponds to quaternion $\alpha$ itself in the case of the Left-Newton and Halley methods and $\alpha+b$ in the case of the Right-Newton method. It is well-known that in a commutative case when the root is simple, the derivative of the iteration function for the Newton method is zero. The proof above resembles the standard proof in the commutative case, except for the need to properly deal with the term $m^{-1} = 2\alpha+b$.
\end{remark}

\begin{remark}
In a general situation, when $\alpha$ and $b$ do not commute, the derivative is zero at exactly two directions; the other two directions have non-zero derivative. This statement is not crucial for our constructions so we neither state it as a theorem, nor prove it.

The key outcome for us is that the local basis near a root can be spanned by two eigenvectors $v_1$ and $v_2$ corresponding to zero, and two eigenvectors $v_3$ and $v_4$ corresponding to non-zero eigenvalues. Every point can be represented as $q=\alpha+x_1v_1+x_2v_2+x_3v_3+x_4v_4$. The effect of applying the function $f$ to it is equivalent to multiplying the vectors by the derivative matrix $D$:
\begin{equation}
f(q) \approx \alpha+D(x_1v_1+x_2v_2+x_3v_3+x_4v_4) =\alpha+\lambda_3 x_3v_3+\lambda_4x_4v_4.
\end{equation}
Thus, the contributions from the two zero eigenvectors are annihilated and the 2D plane spanned by two non-zero eigenvectors, $v_3$ and $v_4$, is {\it invariant}. Of course, here we mean first-order invariance, that is: the image of a point in the plane spanned by $v_3$ and $v_4$, is in the same plane, up to quadratic terms. Geometrically our locally invariant plane is a plane passing through the root $\alpha$ and going in the direction of the two eigenvectors of the derivative matrix corresponding to two non-zero eigenvalues. In proper terms we should consider eigen-directions as belonging to the tangent plane to our quaternion space at the root, but in practical terms we can identify the tangent plane at the root with points near the root.
\end{remark}

\begin{remark}
In some situations the non-zero eigenvalues would be complex and the corresponding two, or even all four, eigenvectors would be complex. In this case, to get the two directions from complex eigenvectors $v_1$ and $v_2$ we consider two real 4D vectors $\Re(v_1)$ (real part) and $\Im(v_1)$ (imaginary part). It is a standard linear algebra result that these two vectors are linearly independent. The decomposition above then should be replaced by $\alpha+x_1\Re(v_1)+x_2 \Im(v_1)+x_3v_3+x_4v_4$ or $\alpha+x_1\Re(v_1)+x_2 \Im(v_1)+x_3\Re(v_3)+x_4\Im(v_3)$.
\end{remark}

\section{Polynomiography with Quaternion Quadratic Polynomials}

Polynomiography is the algorithm visualization of polynomial
equations using iteration functions, see \cite{kalbook}-\cite{kal2005c} where it is defined in the context of complex
polynomial root-finding. An individual image resulting from
polynomiography is called a {\it polynomiograph}. In this section we
describe how we may do polynomiography with quaternion quadratic
polynomials in dimensions 2, 3, and 4. First we discuss how to
construct quadratic quaternions with polynomials having prescribed roots.

\subsection{Quaternion Quadratics with Prescribed Roots}

We wish to construct a quaternion quadratic
\begin{equation} \label{qq0}
P(x)=x^2+bx+c
\end{equation}
having two roots $\alpha, \beta$ from two different conjugacy
classes. From the general construction described in \cite{kalQ}  the corresponding equation would be
\begin{equation} \label{qq1}
P(x)=(x-\gamma)(x- \alpha),
\end{equation}
where
\begin{equation} \label{qq2}
\gamma=(\beta-\alpha)\beta(\beta-\alpha)^{-1}.
\end{equation}
Expanding $P(x)$ and since constants commute with $x$ we get
\begin{equation}
P(x)=x^2 - (\gamma + \alpha)x+ \gamma \alpha.
\end{equation}
Thus $b$ and $c$ can be described in terms of $\alpha, \beta$.

Alternatively, in this special case of a quadratic polynomial with prescribed roots we could derive $b$ and $c$, simply by making use of the fact that
\begin{equation}
\alpha^2+b \alpha +c =0, \quad \beta^2+b \beta +c =0.
\end{equation}
Subtracting the first equation from the second we can solve for $b$ and subsequently for $c$ to get
\begin{equation}
b=(\beta^2-\alpha^2)(\beta-\alpha), \quad c= -(\alpha^2 + b \alpha) =-(\beta^2 + b \beta).
\end{equation}

\subsection{Polynomiography with Quaternion Quadratics}
Polynomiography of a complex quadratic polynomial under Newton or Halley methods
is quite simple looking in the sense that each basin of attraction is the Voronoi cell of a root. However, for quaternion polynomiography it is very different.

One way of associating polynomiography  with a quaternion quadratic
polynomial $P(x)$ is to simply apply polynomiography methods to
$F(x)= P(x) \overline P(x)$. As we have seen in Theorem \ref{thmFx} there is a direct connection between the roots of $P$ and those of $F$.

In what follows we will describe several approaches for generation of a two-dimensional polynomiographs from the  given iteration functions we have considered. To begin with, in order to produce a 2D image, we select a rectangular part of a plane in the quaternion space. We will call it a {\it visualization area}.
For any choice of the visualization area we give: (i) the Left-Newton method polynomiograph; (ii) the Right-Newton method polynomiograph; (iii) the Halley method polynomiograph; (iv) computation times for each of the three methods.  The quaternion polynomial we chose for numerical experiments has roots at
\begin{equation} \label{roots}
\alpha = -1.3+ 2.1{\bf i} +0.17{\bf j} -0.31{\bf k},\quad \beta =
1.4 +0.7{\bf i} -0.23{\bf j} +0.28{\bf k}.
\end{equation}
Our reasoning for this particular choice is because we want the roots to be close enough to the complex plane. As some of our visualization methods use projection into the complex plane we wished to visualize them, at least initially, in a ``good'' situation where the roots and their possible ``shadows'' on the complex plane are not too far apart.  The roots are chosen ``random enough'' to avoid accidental commutativity of coefficients, roots, poles (with respect to iteration functions), etc. The coefficients $b, c$ in $P(x)$ are
\begin{equation}
b=-(0.1+2.6664 {\bf i} +0.5611 {\bf j} +0.0741 {\bf k}), \quad c=-2.9569+ 2.0171 {\bf i}-0.71178 {\bf j} -1.658 {\bf k}.
\end{equation}

Before describing four 2D polynomiography methods we give a definition to help simplify their description.

\begin{definition} Given an iteration function $f(x)$ for approximation of roots of a given quaternion polynomial $P(x)$, the {\it orbit} of a given quaternion $q_0$, called the {\it initial iterate} (or {\it seed}), is the sequence $O(q_0)=\{q_k=f(q_{k-1}): k=1,  \dots \}$. A given algorithm based on iteration of $f(x)$ terminates the orbit for some integer $k(q_0)$.  We call $q_{k(q_0)}$ the {\it terminal iterate} of $q_0$. We refer to any iterate other than the initial and terminal iterates as {\it intermediate iterate}.
\end{definition}

\subsection{Method I:  Tracing Quaternion Iterates}
\label{CPVisualization}

Our first choice of the visualization area is the complex plane.  In this approach we iterate points selected from the complex.  We begin with each points in a rectangular area in the complex plane and consider what happens to the iterates.

For a quaternion $q$, let us denote by ${\Cal C}(q)$ the projection of $q$ onto the complex plane. In other words
\begin{equation}
{\Cal C}(a_1+a_2 {\bf i} + a_3 {\bf j}+ a_4 {\bf k})= a_1+a_2 {\bf i}.
\end{equation}
We consider a complex-plane neighborhood that is big enough to contain the projection of both roots: ${\Cal C}(\alpha)$ and ${\Cal C}(\beta)$. We will choose a square of size $2s\times 2s$, containing the projections. We discretize the square as a grid of points with step size $h$. Every point of the resulting grid is thought of as initial iterate $q_0$, but also as a pixel of the image we produce. We assign a color to the pixel $q_0$ according to the number of iterations it takes for the sequence beginning at $q_0$ to converge to root, or the number of iterations is above a threshold. Here, $q_0\in \mathbb C$ is also considered as an element of $\mathbb H$ and iterated over quaternions using one of the root-finding algorithms $f$ we study:
\begin{equation}
q_{n+1}=f(q_n).\label{ComplexPlane}
\end{equation}
In this method, the initial point of iteration is a complex number but all the following iterations are quaternions. The pixels of the resulting image are initial points of iteration taken from some 2D part of the complex plane $\bc$. We call this method of visualization a {\it complex plane polynomiograph}.

Below we present and discuss the images for complex plane polynomiographs. Every pixel corresponds to a point from the square of the complex plane $[-2.3, 2.3] \times [-2.3,2.3]$.  The images are computed with resolution $1024\times 1024$.   For time measurement purposes we used $100 \times 100$ to make the computations more efficient.  We allow up to 70 iterations for a point. As a criterion to stop the iteration we use either reaching a fixed point or reaching a 2-, 3-, 4- or 5-cycle. We do not check for possibility of a cycle with a longer period.  The stopping criterion used is:
\begin{equation}
\| q_{n+1}-q_n\| \leq 0.01.
\end{equation}

\begin{figure}
\begin{center}
\includegraphics[width=2in]{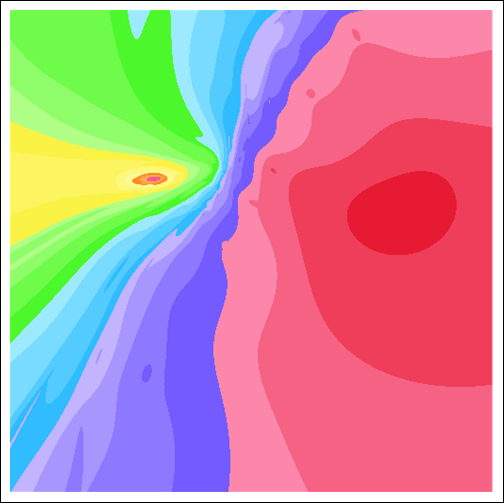}
\includegraphics[width=2in]{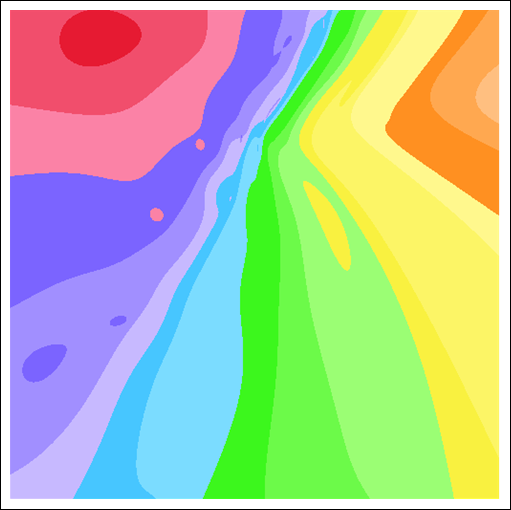}
\includegraphics[width=2in]{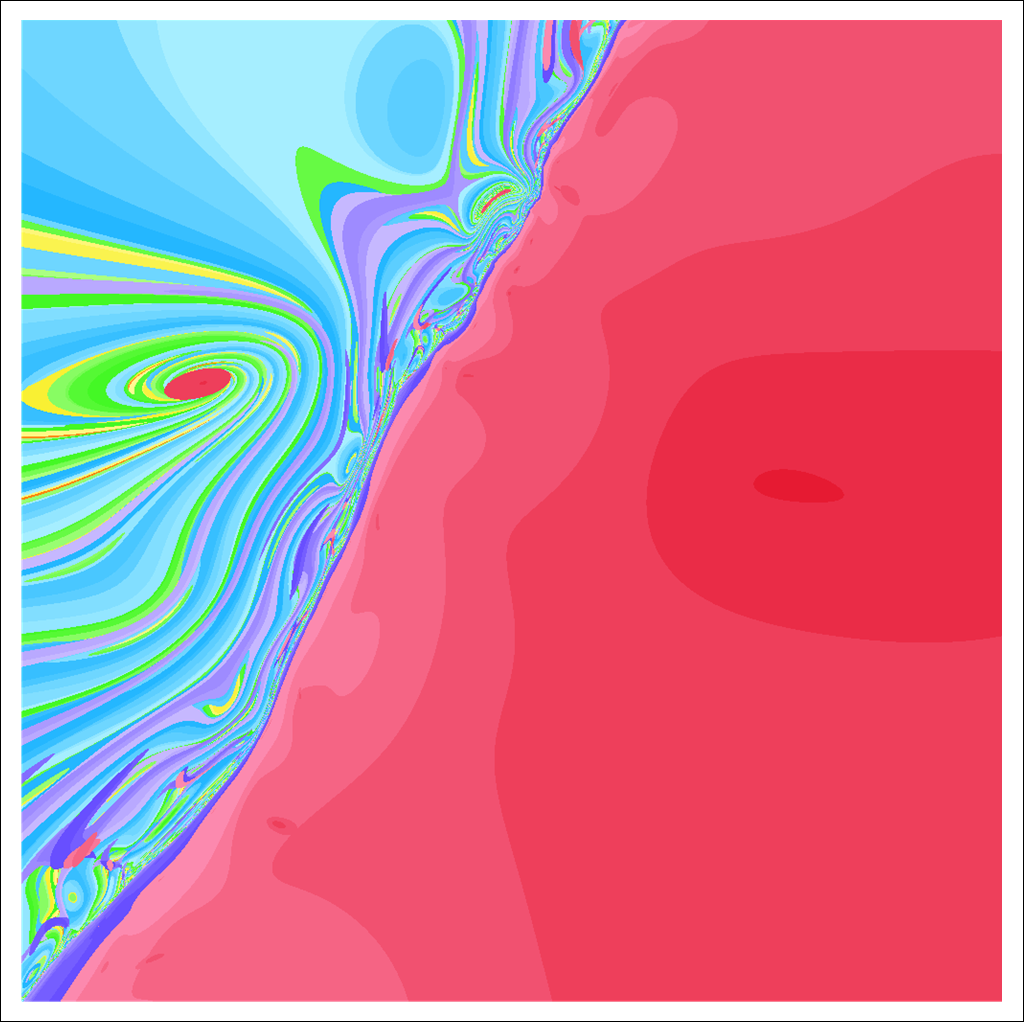}
\caption{Left-Newton, Right-Newton, and Halley method: Complex Plane polynomiograph.}
\label{leftComplex}
\end{center}
\end{figure}

We will do a brief analysis of the images to explain how they help in understanding convergence around the fixed points. In  Fig.~\ref{leftComplex}, left image,  we see two basins of attraction: one is on the left and the other one is on the right. Since the image is based on complex plane visualization, our quaternion roots do not lie on it. Nevertheless, the quaternion roots are close enough to this plane for points near the roots to converge. The points on the right of the image converge to the second root (which is attracting for the case of Left-Newton method). We see a rapid convergence by observing big areas of solid color around the root. The points on the left of the image converge to the first root, but the convergence is much slower. There are no big circle-like areas around that root (strictly speaking, the root's projection to the complex plane). Instead we observe an ellipse-like eccentric structures and horizontal major axis.

Let us explain why for the left root the eccentricity is high and the major axis is horizontal. We recall (Theorem \ref{inv}) that near a root (fixed point) there are two vanishing and two non-vanishing directions. One of the vanishing directions is always the direction of the real numbers ($x$-axis, the horizontal direction on  Fig.~\ref{leftComplex}, left image). The $y$-direction is generally not vanishing. Locally we represent a small deviation from projection ${\cal C} (\alpha)$ of root $\al$ as:
\begin{equation}
{\cal C} (\alpha)+\delta_x + \delta_y {\bf i}.
\end{equation}
Applying our iteration function results in
\begin{equation}
{\cal C} (\alpha)+ O(\delta_x^2) + \lambda \delta_y {\bf i} +O(\delta_y^2),
\end{equation}
where $\lambda$ is the derivative in $y$-direction at the fixed point. The important observation here is that $x$-component of the iteration vanishes. Thus, instead of  a 2D-vector $(\delta_x,\delta_y)$ we deal now with a 2D-vector $(0,\delta_y)$ perhaps multiplied by some number $\lambda$. While $\lambda$ is not necessarily within the unit circle, its magnitude is not particularly large ($|\lambda| \sim 1$). Even if the resulting iterations converge slow due to $\lambda$ being close to $1$, still the points with non-zero first component will certainly converge even slower. Making the first component $0$ with all other things being equal makes point closer to the origin by a factor of $\sqrt 2$. We conclude that the point $(\delta_x,\delta_y)$ takes about the same number of times to converge as $(0,\delta_y)$, which is closer to the origin. Therefore the addition of $x$-component is ``free'' and doesn't add to convergence time (in the first order of approximation). This is the reason we see ellipses near the left fixed point stretched in the horizontal direction.

For the Right-Newton method, middle image in Fig.~\ref{leftComplex}, the left point is attracting --- one can see regular circles of a solid color near it. Notice that the ``shadows'' of the fixed points are now shifted. The projections of the roots would be
\begin{equation}
{\cal C}(\alpha)= -1.3+ 2.1 {\bf i} ,\quad {\cal C}(\beta)=
1.4 +0.7 {\bf i}.
\end{equation}
Our image is determined by the square $[-2.5,2.5]\times [-2.5, 2.5]$. For the Right-Newton method, the actual position of the ``root shadows'' is closer to $-2.0 +2.4i$ and $2.5+2.2i$. For the attracting point on all images in Fig.~\ref{leftComplex} the ``root shadow'' is visible as the darkest red circle. It's not really the projection of the actual root, but instead the points in the complex plane that converge to the actual root fastest. As the three methods behave somewhat differently, it is expected that the fastest converging complex point (the ``shadow'' of the root) would shift slightly, as indeed it does.

We observe a similar behavior for the Halley method, see right image in Fig.~\ref{leftComplex}. Again, the ellipses near the left fixed point are stretched horizontally. The left fixed point now is repelling: the points go away from it in spirals. Two initial points near the fixed point differing by a real number will show similar iteration behavior, and therefore for points in the horizontal direction it takes longer to escape (in a sense, they are ``closer'' to the fixed point).

On these images we see some of the benefits of choosing the complex plane to do visualization: we can observe both roots (although the dark color areas are not really the roots, they are close enough to them) and there is no apparent bias for one root or the other as the algorithm of visualization doesn't distinguish between them. It is worth noting that in the local invariant plane method, we will choose the planes individually for each root. Such there is no common plane to work with both roots in that method. Choosing the complex plane, which for our choice of roots is not too far from them, allows us to see more of a global behavior.

We also see similarities with Newton method for complex numbers: there is a line (somewhat deformed in the quaternion case) going between the roots, roughly half-distance between them. Here, however, the similarities end. While in the complex case the points closest to the left root were going into the left root (and similar for the right root), in the quaternion case only one root is strongly attracting. The other root shows either slow convergence (as for the case of Left and Right-Newton methods for the complex plane) or no convergence at all (Halley method in any plane near the root).\\

\subsection{Method II: Tracing Projection of Intermediate Iterations}
In this visualization we modify the previous approach and
after every iteration project the result down to the complex plane:
\begin{equation}
q_{n+1}={\Cal C}(f(q_n)).\label{ComplexPlaneProj}
\end{equation}
Images produced according to this technique would be called {\it complex plane projection polynomiographs}. While these iterations cannot converge to a true root $\alpha$, it is possible that they converge to the complex projection of it $\mathcal C(\alpha)$. The study and discussion of the images below show that there is indeed a fixed point in the complex plane for this method. It does not coincide, although is close, with the projection of the root.

In fact, all the images obtained via casting a quaternion to a complex number after every iteration (let us call them {\it projection images}) demonstrate presence of fixed points. On the images they correspond to the darkest shade of blue. We need to make clear that the fixed points of projection methods do not coincide with the fixed points of our quaternion iteration function $f$. To start with, the fixed points of quaternion iteration are in the quaternion space and not in the complex plane.
Instead, effectively we deal with a function $g$ from $\bc \to \bc$ constructed as:
\begin{equation}
g:\ \bc \to \bh \xrightarrow{f} \bh \xrightarrow{proj} \bc ,
\end{equation}
where the first arrow is the standard inclusion considering  a complex number as a quaternion.

Nevertheless, the fixed points of the resulting complex iterating functions are close to projections of the quaternion fixed points.
For the projection method, we find the two attracting fixed points, based on where iterations stop, at
\begin{eqnarray*}
\mbox{Left-Newton: }-1.17679+ 2.09022 {\bf i},\ 1.33468+ 0.70161 {\bf i},\\
\mbox{Right-Newton: }-1.17442+ 2.01315 {\bf i},\ 1.3359+ 0.522777 {\bf i} \\
\mbox{Halley: }-1.23475+ 2.05925 {\bf i},\ 1.39796+0.731103 {\bf i}
\end{eqnarray*}
All the images in these two sub-sections are produced for the following area of complex plane: $[-2.3,2.3]\times [-2.3,2.3]$. We remind the reader that the true roots, the quaternion fixed points, are  those in (\ref{roots}).

\begin{figure}
\begin{center}
\includegraphics[width=2in]{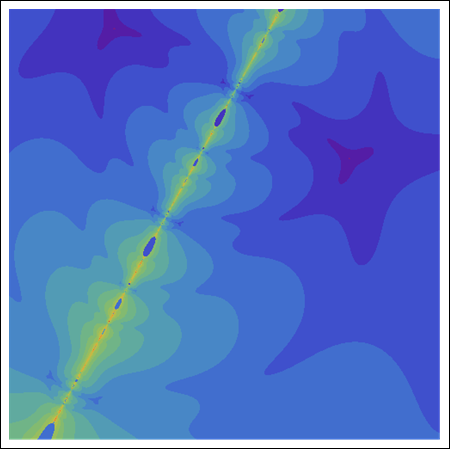}
\includegraphics[width=2in]{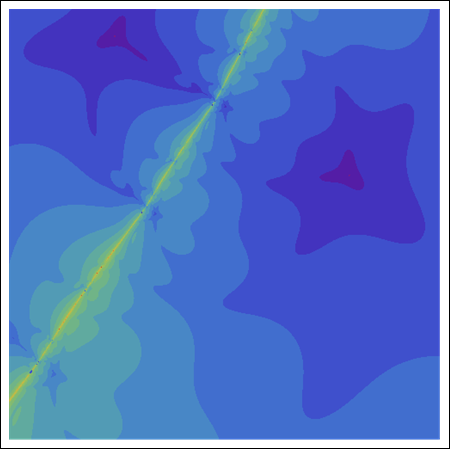}
\includegraphics[width=2in]{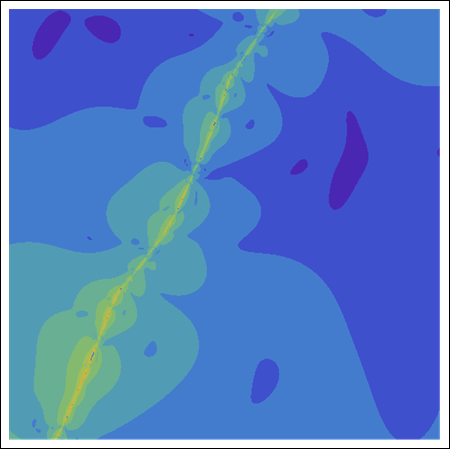}
\caption{Left-Newton, Right-Newton, Halley method: Complex Plane projection.}
\label{leftComplexProjection}
\end{center}
\end{figure}
Now, analyzing images in Fig.~\ref{leftComplexProjection} we observe that they all have two fixed points. The solid areas of color around them show that they are all attracting. It is worth reminding that these, however, are not the fixed points of our quaternion iteration, but rather of complex iteration obtained as a composition of quaternion iteration with projection.

Another common feature of the three images is a line (shown in yellow) separating the two fixed points. It goes, approximately, half way between the fixed points. The fixed points for each of our three methods are somewhat different and so is the line. For instance, a careful reader would notice that the line on left image in Fig.~\ref{leftComplexProjection} ends at the bottom of the image, while the line on the middle image in the same figure for the Right-Newton method ends in the same lower left corner of the image, but on its left side, and not the bottom. In case of Halley method, the ``line'' is somewhat more deformed. Again, we see that Halley method closer resembles the Left-Newton method, than the Right-Newton method.

\subsection{Method III: Tracing Projection of Congruent Intermediate Iterates}

Instead of projection $\Cal C$ from $\bh$ to $\bc$ we
we may consider $\Cal S$-projection, referred as {\it congruency projection}:
\begin{equation}
\Cal S(a_1+a_2 {\bf i} +a_3{\bf j}+a_4 {\bf k})= a_1+\sqrt{a_2^2+a_3^2+a_4^2} {\bf i}.
\end{equation}
The iteration is then defined as follows
\begin{equation}
q_{n+1}={\cal S}(f(q_n)).
\end{equation}

For congruent projection, the fixed points are at
\begin{eqnarray*}
\mbox{Left-Newton: } -1.17651+ 2.13343{\bf i},\ 1.31599+ 0.863941{\bf i},\\
\mbox{Right-Newton: }-1.17642+ 2.06473 {\bf i},\ 1.33545+ 0.687986{\bf i}\\
\mbox{Halley: }-1.23252+2.10451,\ 1.39705+ 0.88163 {\bf i}
\end{eqnarray*}
The effective fixed points of this iteration are close to the ones from the previous subsection. The shapes of borders on the images has changed. Also, in all three images in  Fig. \ref{leftComplexSpherical}  the curve separating the fixed points is curving more: in the previous case it was practically a straight line. Visually, as the areas of dark color are bigger on the first two images in Fig. \ref{leftComplexSpherical}, we observe that the Left-Newton and Right-Newton methods show better global convergence in this case, for the particular choice of visualization area.

\begin{figure}
\begin{center}
\includegraphics[width=2in]{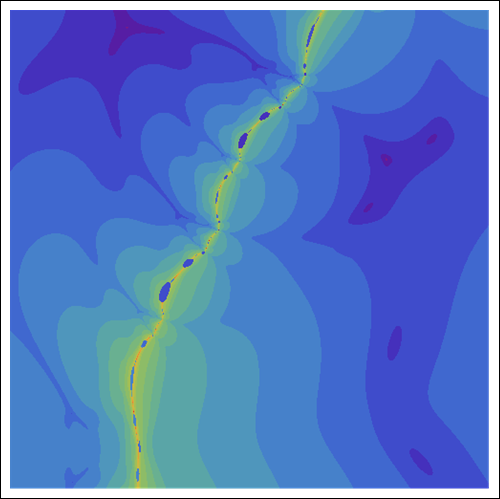}
\includegraphics[width=2in]{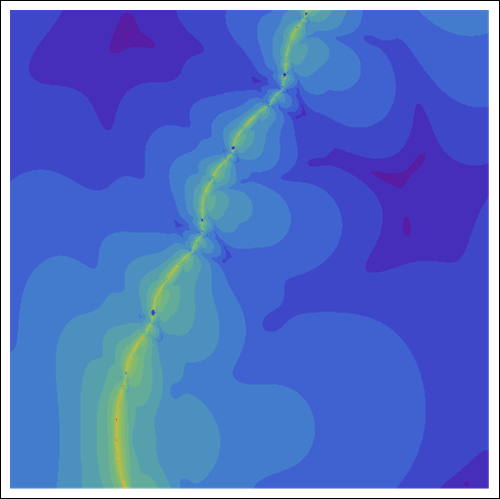}
\includegraphics[width=2in]{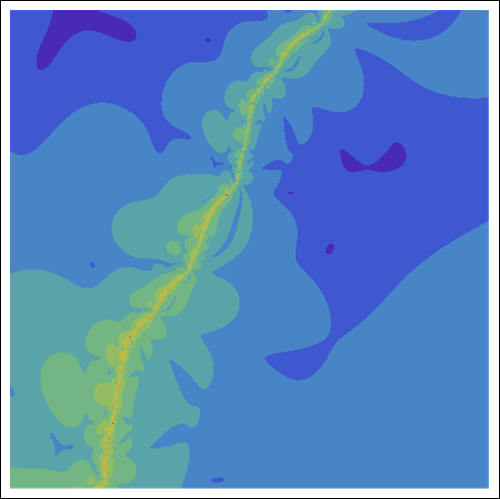}
\caption{Left-Newton, Right-Newton, Halley method: Congruency Projection.}
\label{leftComplexSpherical}
\end{center}
\end{figure}

\subsection{Method IV: Tracing  Iterates in  Locally Invariant Plane}

Choosing complex plane for visualization area has certain advantages as discussed earlier. One, however, can not ignore the fact that the complex plane $\br + \br {\bf i}$ is an arbitrary choice  of a plane in the 4D quaternion space. The only special direction in the quaternion space is the real axis $\br$. Any plane containing $\br$ and one other direction is unlikely to be distinguished from any other such plane. In particular, the choice $\br +\br {\bf i}$ can easily be $\br + \br {\bf j}$ or $\br + \br q$ with any non-real quaternion $q$: all of these planes have one real, commutative, direction and all points from within a plane commute with each other. As it is impossible to choose a ``natural'' 2D plane in the quaternion space, we may look at a ``natural'' choice of a plane for our particular polynomial. A good candidate for such a plane is the locally invariant plane defined in Section \ref{LIPlane}.

As the plane defined in Section \ref{LIPlane} is a) locally invariant; b) contains a fixed point of iterations; and c) uniquely defined at any given root, it is very reasonable to choose the plane as our 2D visualization plane. Thus instead of the complex plane in this method of visualization we work with the local invariant plane. The polynomiographs produced in this plane and discussed in this subsection will be called {\it local plane polynomiographs.}

All other steps in our image producing techniques remain essentially the same: pixels represent points in the (this time) locally invariant plane and the colors indicate the speed of convergence. However, as the locally invariant plane for root $\alpha$ generally will not coincide with the one defined at root $\beta$, we can not have both roots in our plane. We decided to make this unequal role of roots clear in this method by placing the center of the visualization area at the root, at which the plane is considered. The other root is not in the plane, but can be projected down to the plane. For the orientation of the plane we chose the following approach: the horizontal directions on the images in this subsection correspond to the direction from root, e.~g., $\alpha$ to the projection of the other root, in this case, $\beta$ to the locally invariant plane constructed for root $\alpha$. For the complex plane polynomiographs we did not have to be concerned with such issues as the complex plane has ``natural'' orientation: the real numbers go in the horizontal direction and the imaginary numbers $\br i$ go in the vertical direction. For the local plane polynomiographs the horizontal direction points, loosely speaking, to the other root. We chose the size of the visualization area in such a way that the other root is almost on the border of the image.

\begin{figure}
\begin{center}
\includegraphics[width=2in]{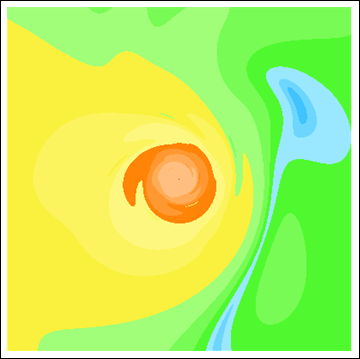}
\includegraphics[width=2in]{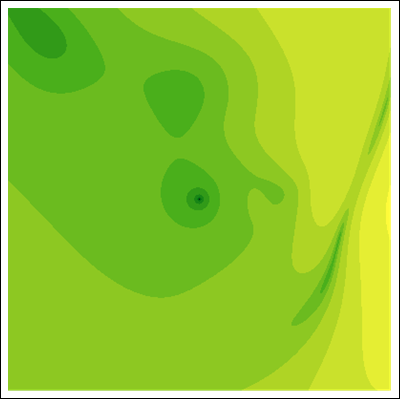}
\includegraphics[width=2in]{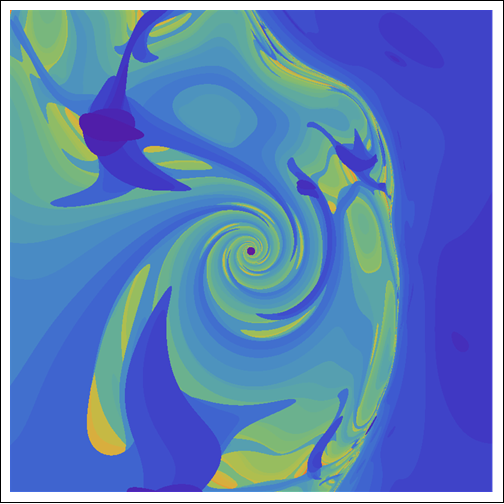}
\caption{Left-Newton, Right-Newton, Halley method: locally invariant plane.}
\label{LeftLocal}
\end{center}
\end{figure}

On left image in Fig.~\ref{LeftLocal} we see our first local plane polynomiograph produced for the Left Newton method. As discussed in Section ~\ref{LIPlane} for our particular choice of the polynomial the derivative matrix in the non-vanishing direction will have two complex eigenvalues with magnitude $|\lambda|>1$. It is worth noting that all the images are produced for root $\alpha$ for which: Left-Newton method is repelling (in our plane), Right-Newton method is attracting, and Halley method is repelling. For the other root, the situation will be the opposite: Left-Newton method is attracting, Right-Newton method is repelling, and Halley method is attracting. The linear terms of the iterations near a fixed point are the same for Left-Newton and Halley methods, thus any statement about stability of one of them automatically extends to the other one. We are not discussing a potential exception to this rule, when the non-vanishing eigenvalue $\lambda$ is such that $|\lambda|=1$ and the corresponding fixed point is not hyperbolic (in the plane). In this case the attraction of the fixed point is affected by quadratic terms and the Left-Newton stability could be different from stability for Halley method.

Due to the fact that the eigenvalue $\lambda$ is complex and  $|\lambda|>1$ we should expect points escaping the neighborhood of the fixed point (center of the image) in spirals: as indeed we observe on the left image in Fig.~\ref{LeftLocal}. Approximately half way between the roots, the area of convergence to the right root begins. For the particular choice of coloring scheme, the central root, repelling in this situation, is yellow and the attracting, right, root is green.

The Right-Newton method is visualized in the middle image in Fig.~\ref{LeftLocal}. It is worth noting that the locally invariant plane depends on the method: the plane for the Left-Newton method is different from the Right-Newton method plane, even at the same root. Thus, the first two images in Fig.~\ref{LeftLocal}  visualize different methods in different planes: the only common part of them is the root at the center. The other ``almost common'' part is the second root: here, however, we should remember that the second root is not in the locally invariant plane, we see only its projection to the plane. The first root, the fixed point we study, is in the center of the middle image in Fig.~\ref{LeftLocal} (dark green for the choice of colors). The presence of the second root can be seen at the midpoint of the right edge, bright yellow colors --- exactly where we should expect it according to our orientation of the locally invariant plane and the choice of the size for the visualization area. This particular main root, the fixed point, is attracting for the case of the Right-Newton method.

Although the locally invariant plane generally depends on the choice of a method, the plane for the Left-Newton method coincides with the one for Halley method as both methods have the same linear terms, thus having the same derivative matrix which alone dictates the choice of the invariant plane. Therefore, the image we turn our attention to, the right image in Fig.~\ref{LeftLocal}, shows exactly the same area as that of the Left-Newton method.

The image demonstrates an intriguing level of complexity perhaps somewhat unexpected for a quadratic polynomial. Although all our images are produced based on iterations, none of them exhibit fractal behavior. Nevertheless, the images are far from straightforward, as we can see for the case of Halley method. At the center of the image we observe a repelling fixed point with a complex multiplier, thus forcing the points to leave vicinity of the fixed point in spirals.  On the right, at the middle of the right edge, we see the second root of our polynomial, or, to be precise, its projection on our locally invariant plane. The root is attracting everything at the right part of the image. The tree-like structure at the top-left corner deserves a special attention. When constructing this image we were checking for 2-, 3-, 4- and 5-cycles. The tree-like structure is, apparently, the area converging to the 3-cycle:
\begin{equation}
-0.9+1.3{\bf i} -0.55{\bf j}+0.6{\bf k}, \quad -1+{\bf i}+1.5{\bf j}+0.4{\bf k}, \quad  -1.4+0.5 {\bf i}-0.4{\bf j}-0.9{\bf k}
\end{equation}
We did not detect any other cycles in the area we studied. This 3-cycle attracts about 70\% of the image points. Without checking for 3-cycles, the image would look like Fig.~\ref{HalleyInv01}
\begin{figure}
\begin{center}
\includegraphics[width=2in]{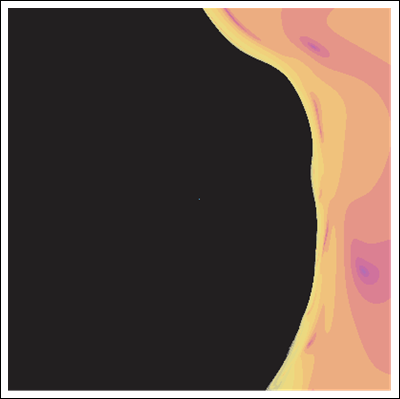}
\caption{Halley method in Fig.~\ref{LeftLocal} without cycle checking.}
\label{HalleyInv01}
\end{center}
\end{figure}

The noticeable difference between Halley iteration in  Figs.~\ref{LeftLocal} and \ref{HalleyInv01}gives us a certain justification for using images as a tool to study iterations in numerical analysis. It reveals existence of different areas of attraction, the speed of convergence, reflected in colors, and the character of convergence as well as divergence from the repelling fixed point (the spiral at the center).

In conclusion of our discussion for the images, it is worth noting that all the images presented for a certain method, e.~g. Halley method are found for the same polynomial. Thus, by gradually changing the plane of visualization we will somehow go from Fig.~\ref{leftComplex}  to Fig.~\ref{LeftLocal}. Exactly how this deformation is happening is perhaps a question for further studies.

\subsection{Time Measurements}

In the previous sections we discussed iterative methods and the expectations about speed of its convergence. The numerical results in the table confirm those expectations. In particular, we discussed that one of the roots is attracting for the Left-Newton method when we iterate in the locally invariant plane and repelling for the Right-Newton method. For the second root the roles are switched. The time measurement indeed show  that the Left-Newton method takes 60.5 seconds to compute the image and the Right-Newton method takes 39.9 seconds for the first root.
 For the second root we expect the numbers to switch and indeed they do: 39.6 seconds for the Left-Newton method and 61.0 seconds for the Right-Newton methods.

Another observation we can make is that apparently Mathematica's quaternion package works faster when the quaternions are complex numbers: the methods that project to complex numbers after every step finish significantly faster.  The operations (addition and multiplication) performed are quaternion operations but in these methods many quaternion operands happen to be complex numbers. Of course, the disadvantage of these methods is they never find a real quaternion solution. At best, a rough approximation to the projection of the quaternion to the complex plane.

\begin{table}
\begin{center}
\begin{tabular}{|c|c|c|c|}
\hline
{} Algorithm &Left-Newton & Right-Newton & Halley\\
\hline
 Method I: Tracing Quaternion Iterates & 68.4 & 84.4 & 132.9\\
Method II: Tracing Intermediate Projection & 31.6 & 30.7 & 41.8\\
Method III: Tracing Congruence Projection & 35.6 & 34.1 & 44.0\\
Method IV: Tracing Iterates in Invariant Plane  of $\alpha$ & 60.5 & 39.9 & 170.9\\
Method IV: Tracing Iterates in Invariant Plane  of $\beta$ & 39.6 & 61.0 & 74.2 \\
\hline
\end{tabular}
\caption{Time measurements (in seconds).}
\label{TimeMeasurement}
\end{center}
\end{table}

We remark that solving $F(x)$ via Newton's method to generate a polynomiograph of the same size requires only 2.2 seconds, shown in Fig.~ \ref{FigF} (left).   This suggests for practical purposes it is more efficient to solve $F(x)$ than using iterations over the quaternion space to solve $P(x)$.

\begin{figure}
\begin{center}
\includegraphics[width=2in]{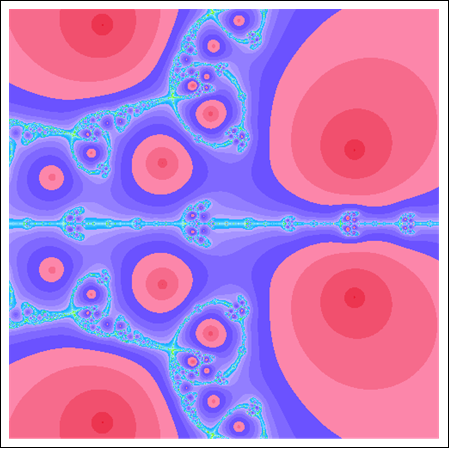}
\includegraphics[width=2in]{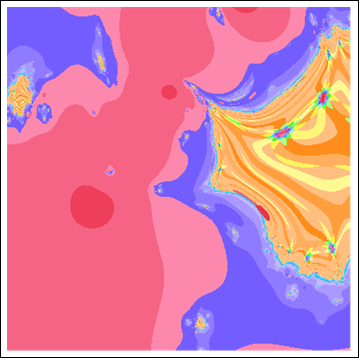}
\caption{Polynomiograph of Newton's method applied to F(x) (left).
Polynomiograph of a hybrid method, iterations of Newton with $P(x)$ and $F(x)$ (right).}
\label{FigF}
\end{center}
\end{figure}

\subsection{Hybrid 2D Polynomiography}
The connections between $P(x)$ and $F(x)=P(x)\overline P(x)$  also suggests hybrid algorithms. These in turn would lead to 2D polynomiography. For instance, consider a variation of using an iteration function $f$
in the quaternion space where the iterates are computed as applied to $P(x)$. If convergence
is not observed the algorithm could switch to applying $f$
to $F(x)$, possibly lifting it back to the quaternion space. More specifically,
denote the iteration of $f$  corresponding to $P(x)$ and
$F(x)$ by $f_P(x)$ and $f_F(x)$, respectively. Given a quaternion
$x_{k-1}$ as the current iterate and assuming that $|P(x_{k-1})|$
exceeds a desired tolerance $\epsilon$, we compute the next iterate
$x_k$ as follows: Let
\begin{equation}
y_k= {\cal S}(f_P(x_{k-1})),
\end{equation}
then if $|\overline P(y_k)| > \epsilon$ and $|\overline P(\overline y_k)| > \epsilon$, lift $y_k$ to a quaternion using Theorem \ref{Approx}.
Specifically, set $x_k$ as
\begin{equation}
x_k = \overline P(y_k) y_k  \overline P(y_k)^{-1} \quad \text{or} \quad
x_k=\overline P( \overline y_k)  \overline y_k  \overline P(
\overline y_k)^{-1}.
\end{equation}

Fig. \ref{FigF} (right) shows a polynomiograph with respect to performing two Newton iterations over quaternions with respect to $P(x)$,  followed by  projection down to complex and one Newton iteration with  respect to $F(x)$, and repetition of the process without lifting the result to quaternions.

\subsection{3D and 4D Polynomiography}

One may also employ an iterative method such as Newton or Halley
directly in the quaternion space and develop a corresponding coloring.
The behavior of these methods, including their convergence
remains to be further investigated.  We may consider a quadratic equation with prescribed zeros and then keep track of the behavior of the iterative methods directly in three of the four dimensions.  4D polynomiography can also be carried out by recording the history of the iterates of an iterative process, followed by projection
into a 3D subspace.



\bigskip


\end{document}